\documentclass[11pt,a4paper,leqno]{amsart}
\usepackage{calrsfs}  

\newcommand{\ud}{\mathrm{d}}

\renewcommand{\d}{\mathrm{d}}

\renewcommand{\epsilon}{\varepsilon}

\renewcommand{\P}{\mathbb{P}}
\newcommand{\E}{\mathbb{E}}

\newcommand{\R}{\mathbb{R}}

\begin{document}


\theoremstyle{definition}
\newtheorem{dfn}{Definition}
\newtheorem{ass}{Assumption}
\theoremstyle{plain}
\newtheorem{thm}{Theorem}
\newtheorem{pro}{Proposition}
\newtheorem{cor}{Corollary}
\newtheorem{lmm}{Lemma}
\theoremstyle{remark}
\newtheorem{rem}{Remark}
\newtheorem{exa}{Example}



 
\title[H\"older Continuity of Gaussian Processes]{Necessary and Sufficient Conditions for H\"older Continuity of Gaussian Processes}
\date{\today}

\author[Azmoodeh]{Ehsan Azmoodeh}
\address{Ehsan Azmoodeh\\ Facult\'e des Sciences, de la Technologie et de la Communication,  Universit\'e du Luxembourg\\
P.O. Box L-1359, LUXEMBOURG} 

\author[Sottinen]{Tommi Sottinen}
\address{Tommi Sottinen\\ Department of Mathematics and Statistics\\ University of Vaasa\\ P.O. Box 700\\ FIN-65101 Vaasa\\ FINLAND}

\author[Viitasaari]{Lauri Viitasaari}
\address{Lauri Viitasaari\\ Department of Mathematics and System Analysis, Aalto University School of Science, Helsinki\\
P.O. Box 11100, FIN-00076 Aalto,  FINLAND} 

\author[Yazigi]{Adil Yazigi}
\address{Adil Yazigi\\ Department of Mathematics and Statistics\\ University of Vaasa\\ P.O. Box 700\\ FIN-65101 Vaasa\\ FINLAND}

\begin{abstract}
The continuity of Gaussian processes is extensively studied topic and it culminates in the Talagrand's notion of majorizing measures that gives complicated necessary and sufficient conditions.  In this note we study the H\"older continuity of Gaussian processes.  It turns out that necessary and sufficient conditions can be stated in a simple form that is a variant of the celebrated Kolmogorov--\v{C}entsov condition.
\end{abstract}

\thanks{A. Yazigi was funded by the Finnish Doctoral Programme
in Stochastics and Statistics.}

\keywords{Gaussian processes,
H\"older continuity,
Kolmogorov--\v{C}entsov condition,
self-similar processes.
}

\subjclass[2010]{60G15, 60G17, 60G18}

\maketitle


 
\section{Introduction}

In what follows $X$ will always be a centered Gaussian process on the interval $[0,T]$. For a centered Gaussian family $\xi=(\xi_\tau)_{\tau\in \mathbb{T}}$ we denote
\begin{eqnarray*}
d_\xi^2(\tau,\tau') &:=& \E[(\xi_\tau- \xi_{\tau'})^2], \\
\sigma^2_\xi(\tau) &:=& \E[\xi_\tau^2].
\end{eqnarray*}

To put our result in context, we briefly recall the essential results of Gaussian continuity. 

One of the earliest results is a sufficient condition due to Fernique \cite{Fernique}:
\emph{Assume that for some positive $\epsilon$, and $0 \le s \le t \le \epsilon$, there exists a nondecreasing function $\Psi$ on $[0,\epsilon]$ such that $\sigma^{2}_X(s,t) \le \Psi^2(t-s)$ and 
\begin{equation}\label{Fernique-integral}
 \int_0^\epsilon \frac{\Psi(u)}{u \sqrt{\log u}} \,\d u < \infty.
\end{equation}
Then $X$ is continuous.}
The finiteness of Fernique integral (\ref{Fernique-integral}) is not  necessary for the continuity. Indeed, cf. \cite[Sect. 5]{Marcus-Shepp} for a counter-example.

Dudley \cite{Dudley1, Dudley2} found a sufficient condition for the continuity by using \emph{metric entropy}. Let $N(\epsilon):=N([0,T],d_X,\epsilon)$ denote the minimum number of closed balls of radius $\epsilon$ in the (pseudo) metric $d_X$ needed to cover $[0,T]$. 
\emph{If 
\begin{equation}\label{dudley-integral}
\int_{0}^\infty \sqrt{\log N(\epsilon)} \,\ud \epsilon < \infty,
\end{equation}
then $X$ is continuous.}
Like in the case of the Fernique's condition, the finiteness of the Dudley integral (\ref{dudley-integral}) is not necessary for continuity, cf. \cite[Ch 6.]{marcus-rosen}. However, for stationary processes (\ref{dudley-integral}) is necessary and sufficient.

Finally, necessary and sufficient conditions were obtained by Talagrand \cite{talagrand}.  Denote $B_{d_X}(t,\epsilon)$ a ball with radius $\epsilon$ at center $t$ in the metric $d_X$. A probability measure $\mu$ on $([0,T],d_X)$ is called a \emph{majorizing measure} if 
\begin{equation}\label{Talagrand-condition}
\sup_{t \in [0,T]} \int_{0}^{\infty} \sqrt{\log \frac{1}{\mu\left(B_{d_X}(t,\epsilon)\right)}} \,\ud \epsilon < \infty.
\end{equation}
\emph{The Gaussian process $X$ is continuous if and only if there exists a majorizing measure $\mu$ on $([0,T],d_X)$ such that
$$
\lim_{\delta \to 0} \sup_{t \in [0,T]} \int_{0}^{\delta} \sqrt{\log \frac{1}{\mu(B_{d_X}(t,\epsilon))}} \,\ud \epsilon = 0.
$$
}

\section{Main Theorem}

The Talagrand's necessary and sufficient condition (\ref{Talagrand-condition}) for the continuity of a Gaussian process is rather complicated. In contrast, the general Kolmogorov--\v{C}entsov condition for continuity is very simple. It turns out that for Gaussian processes the Kolmogorov--\v{C}entsov condition is very close to being necessary for H\"older continuity: 

\begin{thm}\label{thm:holder}
The Gaussian process $X$ is H\"older continuous of any order $a<H$ i.e.
\begin{equation}\label{Holder}
|X_t - X_s| \leq C_\epsilon |t-s|^{H-\epsilon}, \quad \mbox{ for all } \epsilon>0
\end{equation}
if and only if there exists constants $c_\epsilon$ such that 
\begin{equation}
\label{var-estimate}
d_X(t,s) \le c_\epsilon |t-s|^{H-\epsilon}, \quad \mbox{ for all } \epsilon>0.
\end{equation}

Moreover, the random variables $C_\epsilon$ in (\ref{Holder}) satisfy
\begin{equation}
\label{H-moment}
\E\left[\exp\left(aC_\epsilon^\kappa\right)\right]<\infty
\end{equation}
for any constants $a\in\R$ and $\kappa<2$; and also for $\kappa=2$ for small enough positive $a$. In particular, the moments of all orders of $C_\epsilon$ are finite.
\end{thm}

The differences between the classical Kolmogorov--\v{C}entsov continuity criterion and Theorem \ref{thm:holder} are: (i) Theorem \ref{thm:holder} deals only with Gaussian processes, (ii) there is an $\epsilon$-gap to the classical Kolmogorov--\v{C}entsov condition and (iii) as a bonus we obtain that the H\"older constants $C_\epsilon$ must have light tails by the estimate (\ref{H-moment}). Note that the $\epsilon$-gap cannot be closed.  Indeed, let 
$$
X_t = f(t)B_t,
$$
where $B$ is the fractional Brownian motion with Hurst index $H$ and $f(t) = (\log\log 1/t)^{-1/2}$. Then, by the law of the iterated logarithm due to Arcones \cite{Arcones}, $X$ is H\"older continuous of any order $a<H$, but (\ref{var-estimate}) does not hold without an $\epsilon>0$. 

The proof of the first part Theorem \ref{thm:holder} is based on the classical Kolmogorov--\v{C}entsov continuity criterion and the following elementary lemma:

\begin{lmm}\label{lmm:holder}
Let $\xi=(\xi_\tau)_{\tau\in\mathbb{T}}$ be a centered Gaussian family.
If $\sup_{\tau\in\mathbb{T}} |\xi_\tau| < \infty$ then $\sup_{\tau\in\mathbb{T}} \E[\xi_\tau^2]<\infty$. 
\end{lmm}

\begin{proof}
Since $\sup_{\tau\in\mathbb{T}} |\xi_\tau|<\infty$, $\P[\sup_{\tau\in\mathbb{T}} |\xi_\tau|<x]>0$ for a large enough $x\in\R$. Now, for all $\tau\in\mathbb{T}$, we have that
\begin{eqnarray*}
\P\left[\sup_{\tau\in\mathbb{T}} \left|\xi_\tau\right|<x\right]
&\le& \P\left[\left|\xi_\tau\right| < x\right] \\
&=& \P\left[\left|\frac{\xi_\tau}{\sigma_\xi(\tau)}\right| < \frac{x}{\sigma_\xi(\tau)}\right] \\
&=& \frac{2}{\sqrt{2\pi}}\int_{0}^{x/\sigma_\xi(\tau)} e^{-\frac12 z^2}\, \d z \\
&\le& \frac{2}{\sqrt{2\pi}} \frac{x}{\sigma_\xi(\tau)}.
\end{eqnarray*}
Consequently,
$$
\sigma_\xi^2(\tau) \le \frac{2x^2}{\pi \P\left[\sup_{\tau\in\mathbb{T}} \left|\xi_\tau\right|<x\right]^2},
$$
and the claim follows from this.
\end{proof}

The second part on the exponential moments of the H\"older constants of Theorem \ref{thm:holder} follows from the following Garsia--Rademich--Rumsey inequality \cite{Garsia-Rademich-Rumsey}. Let us also note, that this part is intimately connected to the Fernique's theorem \cite{Fernique2} on the continuity of Gaussian processes.

\begin{lmm}
\label{lma:GRR}
Let $p\geq 1$ and $\alpha>\frac{1}{p}$. Then there exists a constant $c=c_{\alpha,p}>0$ such that for any $f\in C([0,T])$ and for all $0\leq s,t\leq T$ we have
$$
{|f(t)-f(s)|}^p \leq c T^{\alpha p -1}|t-s|^{\alpha p  -1}\int_0^T\!\!\int_0^T \frac{|f(x)-f(y)|^p}{|x-y|^{\alpha p +1}}\,\d x\d y.
$$
\end{lmm}
\begin{proof}[Proof of Theorem \ref{thm:holder}]
The if part follows from the Kolmogorov--\v{C}entsov continuity criterion. For the only-if part assume that $X$ is H\"older continuous of order $a=H-\epsilon$, i.e.
$$
\sup_{t,s\in[0,T]} \frac{\left|X_t-X_s\right|}{|t-s|^{H-\epsilon}} < \infty.
$$
Define a family $\xi = (\xi_{t,s})_{(t,s)\in [0,T]^2}$ by setting
$$
\xi_{t,s} = \frac{X_t-X_s}{|t-s|^{H-\epsilon}}.
$$
Since $\xi$ is a centered Gaussian family that is bounded by the H\"older continuity of $X$, we obtain, by Lemma \ref{lmm:holder}, that $\sup_{(t,s)\in [0,T]^2} \sigma_\xi^2(t,s) < \infty$. This means that 
$$
\sup_{t,s\in [0,T]}\frac{d_X^2(t,s)}{|t-s|^{2H-2\epsilon}} < \infty,
$$
or
$$
d_X(t,s) \le C_\epsilon|t-s|^{H-\epsilon}.
$$

The property (\ref{H-moment}) follow from the Garsia--Rademich--Rumsey inequality of Lemma \ref{lma:GRR}. Indeed, by choosing $\alpha = H - \frac{\epsilon}{2}$ and $p=\frac{2}{\epsilon}$ we obtain
$$
|X_t - X_s| \leq c_{H,\epsilon} T^{H-\epsilon}|t-s|^{H-\epsilon}\xi,
$$
where
\begin{equation}
\label{H-RV}
\xi = \left(\int_0^T\!\!\int_0^T \frac{|X_u - X_v|^{\frac{2}{\epsilon}}}{|u-v|^{\frac{2H}{\epsilon}}}\,\d u \d v\right)^{\frac{\epsilon}{2}}.
\end{equation}
Let us first estimate moments of $\xi$. First we recall the fact that for a Gaussian random variable $Z\sim\mathcal{N}(0,\sigma^2)$ and any number $q>0$ we have
$$
\E\left[|Z|^q\right] = \sigma^q \, \frac{2^{\frac{q}{2}}\Gamma\left(\frac{q+1}{2}\right)}{\sqrt{\pi}},
$$
where $\Gamma$ denotes the Gamma function. 
Let now $\delta<\frac{\epsilon}{2}$ and $p\geq \frac{2}{\epsilon}$. By Minkowski inequality and estimate (\ref{var-estimate}) we obtain
\begin{eqnarray*}
\E\left[|\xi|^p\right] &\leq& \left(\int_0^T\!\!\int_0^T \frac{\left(\E|X_u - X_v|^p\right)^{\frac{2}{p\epsilon}}}{|u-v|^{\frac{2H}{\epsilon}}}\d v\d u\right)^{\frac{p\epsilon}{2}}\\
&\leq& \left(\int_0^T\!\!\int_0^T \frac{\left(c_p c_\delta|u-v|^{p(H-\delta)}\right)^{\frac{2}{p\epsilon}}}{|u-v|^{\frac{2H}{\epsilon}}}\,\d v\d u\right)^{\frac{p\epsilon}{2}}\\
&=&c_p c_\delta 2^{\frac{p\epsilon}{2}}\left(\int_0^T\int_0^u (u-v)^{-\frac{2\delta}{\epsilon}}\d v\d u\right)^{\frac{p\epsilon}{2}}\\
&=& c_p c_\delta 2^{\frac{p\epsilon}{2}}\left(\frac{\epsilon}{2\delta}\right)^{\frac{p\epsilon}{2}}\left(1-\frac{\epsilon}{2\delta}\right)^{\frac{p\epsilon}{2}}T^{q(\epsilon-\delta)},
\end{eqnarray*}
where $c_\delta$ is the constant from (\ref{var-estimate}) and
\begin{equation}
\label{q-constant}
c_q =\frac{2^{\frac{q}{2}}\Gamma\left(\frac{q+1}{2}\right)}{\sqrt{\pi}}.
\end{equation}
Hence, we may take 
$$
C_\epsilon = c_{H,\epsilon}T^{H-\epsilon}\xi,
$$
where $c_{H,\epsilon}$ is the constant from Garsia--Rademich--Rumsey inequality and $\xi$ is given by (\ref{H-RV}). Moreover, for any $p\geq \frac{2}{\epsilon}$ and any $\delta<\frac{\epsilon}{2}$ we have estimate
$$
\E\left[|\xi|^p\right] \leq c_p c_\delta 2^{\frac{p\epsilon}{2}}\left(\frac{\epsilon}{2\delta}\right)^{\frac{p\epsilon}{2}}\left(1-\frac{\epsilon}{2\delta}\right)^{\frac{p\epsilon}{2}}T^{q(\epsilon-\delta)}.
$$
Consequently, 
$$
\E\left[|C_\epsilon|^p\right] \leq c^p \Gamma\left(\frac{p+1}{2}\right)
$$
for some constant $c=c_{\epsilon,\delta,T}$. Thus, by plugging in (\ref{q-constant}) to the series expansion of the exponential we obtain
$$
\E\left[\exp\left(aC_\epsilon^\kappa\right)\right]
\le 
\sum_{j=0}^\infty a^jc^{\kappa j}\frac{\Gamma\left(\frac{\kappa j+1}{2}\right)}{\Gamma\left(j+1\right)}. 
$$
So, to finish the proof we need to show that the series above converges.
Now, by the Stirling's approximation 
$$
\Gamma(z) = \frac{\sqrt{2\pi}}{\sqrt{z}}\left(\frac{z}{e}\right)^z\left(1+O(1/z)\right),
$$
we obtain (the constant $c$ may vary from line to line)
\begin{eqnarray*}
\frac{\Gamma\left(\frac{\kappa j+1}{2}\right)}{\Gamma\left(j+1\right)} 
&\sim& 
\frac{\left(\frac{\kappa j+1}{2}\right)^{-\frac{1}{2}}e^{-\frac{\kappa j +1}{2}}\left(\frac{\kappa j +1}{2}\right)^{\frac{\kappa j +1}{2}}}{\left(j+1\right)^{-\frac{1}{2}}e^{-j-1}\left(j+1\right)^{j+1}}\\
&\le& 
c^j \frac{1}{\sqrt{j+1}}\frac{\left(\kappa j +1\right)^{\frac{\kappa j}{2}}}{\left(j+1\right)^{j}}\\
&\le& 
c^j \frac{1}{\sqrt{j+1}}\frac{\left(2j +2\right)^{\frac{\kappa j}{2}}}{\left(j+1\right)^{j}}\\
&=& 
(2c)^j\frac{1}{\sqrt{j+1}} (j+1)^{\left(\frac{\kappa}{2}-1\right)j}
\end{eqnarray*}
which is clearly summable since $\kappa < 2$. If $\kappa=2$, then in the approximation above we obtain that 
$
\Gamma\left(\frac{2 j+1}{2}\right)/\Gamma\left(j+1\right) \sim c^j
$
for some constant $c$. Hence, depending on constant $c_{\epsilon,\delta,T}$, we obtain that
$
\E\left[\exp\left(aC_\epsilon^2\right)\right] < \infty
$
for small enough $a>0$.
\end{proof}

\section{Applications and Examples}

\subsection*{Stationary-Increment Processes} 

This case is simple:

\begin{cor}\label{cor:holder-si}
If $X$ has stationary increments then it is H\"older continuous of any order $a<H$ if and only if
$$
\sigma^2_X(t) \le c_\epsilon t^{2H-\epsilon}, \quad \mbox{for all } \epsilon>0.
$$
\end{cor}

\subsection*{Stationary Processes}

For a stationary process $\E[X_tX_s]=r(t-s)$, where, by the Bochner's theorem, 
$$
r(t)= \int_{-\infty}^{\infty} e^{i \lambda t} \,\Delta(\ud \lambda),
$$
where $\Delta$, the \emph{spectral measure} of $X$, is finite and symmetric.
Since now
$$
d_X^2(t,s) = 2\big(r(0)-r(t-s)\big)
$$
we have the following corollary.

\begin{cor}
If $X$ is stationary with spectral measure $\Delta$ then it is H\"older continuous of any order $a <H$ if and only if
$$
\int_0^\infty \big(1-\cos(\lambda t)\big) \,\Delta(\d\lambda) \le c_\epsilon t^{2H-\epsilon}
\quad \mbox{ for all } \epsilon>0.
$$
\end{cor}

\subsection*{Fredholm Processes}

A bounded process can be viewed as an $L^2([0,T])$-valued random variable.  Hence, the covariance operator admits a square root with kernel $K$, and we may represent $X$ as a \emph{Gaussian Fredholm process}:
\begin{equation}\label{eq:fredholm}
X_t = \int_0^T K(t,s)\, \d W_s,
\end{equation}
where $W$ is a Brownian motion and $K\in L^2([0,T]^2)$.

\begin{cor}\label{cor:fredholm}
A Gaussian process $X$ is H\"older continuous of any order $a<H$ if and only if it admits the representation (\ref{eq:fredholm}) with $K$ satisfying
$$
\int_0^T \left|K(t,u)-K(s,u)\right|^2\, \d u \le c_\epsilon |t-s|^{2H-\epsilon} \quad \mbox{ for all } \epsilon>0.
$$
\end{cor}

\begin{pro}\label{pro:fredholm}
Let $X$ be Gaussian Fredholm process with kernel $K$. 
\begin{enumerate}
\item
If for every $\epsilon>0$ there exists a function $f_\epsilon\in L^2([0,T])$ such that
$$
|K(t,u)-K(s,u)| \le f_\epsilon(u)|t-s|^{H-\epsilon}
$$
then $X$ is H\"older continuous of any order $a<H$.
\item
If $X$ is H\"older continuous of any order $a<H$ then
$$
f_\epsilon := \liminf_{s\to t} \frac{\left|K(t,\cdot)-K(s,\cdot)\right|^2}{|t-s|^{2H-\epsilon}} \in L^1([0,T])
$$
\end{enumerate}
\end{pro}

\begin{proof}
The first part follows from Corollary \ref{cor:fredholm}. Consider then the second part and assume that $X$ is H\"older continuous of any order $a<H$ and 
$$
\liminf_{s\to t} \frac{\left|K(t,\cdot)-K(s,\cdot)\right|^2}{|t-s|^{2H-\epsilon}} \notin L^1([0,T]).
$$
By Corollary \ref{cor:fredholm} we know that
$$
\int_0^T \frac{\left|K(t,\cdot)-K(s,\cdot)\right|^2}{|t-s|^{2H-\epsilon}}\d u \leq c_\epsilon.
$$
On the other hand, by Fatou Lemma we have
$$
\liminf_{s\to t} \int_0^T \frac{\left|K(t,\cdot)-K(s,\cdot)\right|^2}{|t-s|^{2H-\epsilon}}\d u \geq \int_0^T \liminf_{s\to t} \frac{\left|K(t,\cdot)-K(s,\cdot)\right|^2}{|t-s|^{2H-\epsilon}}\d u = \infty
$$
which is a contradiction.
\end{proof}

\subsection*{Volterra Processes}

A Fredholm process is a \emph{Volterra process} if its kernel $K$ satisfies $K(t,s)=0$ if $s>t$. In this case Corollary \ref{cor:fredholm} becomes:

\begin{cor}\label{cor:volterra}
A Gaussian Volterra process $X$ with kernel $K$ is H\"older continuous of any order $a<H$ is and only if, for all $s<t$ and $\epsilon>0$
\begin{enumerate}
\item
$\int_s^t K(t,u)^2\, \d u \le c_\epsilon |t-s|^{2H-\epsilon}$, 
\item
$\int_0^s |K(t,u)-K(s,u)|^2\, \d u \le c_\epsilon |t-s|^{2H-\epsilon}$. 
\end{enumerate}
\end{cor}

By \cite[p. 779]{Alos} the following is a sufficient condition:

\begin{pro}
Let $X$ be a Gaussian Volterra process with kernel $K$ that satisfies
\begin{enumerate} 
\item $\int_s^t K(t,u)^2\, \d u \leq c(t-s)^{2H},$
\item $K(t,s)$ is differentiable in $t$ and $\left|\frac{\partial K}{\partial t}(t,s)\right| \leq c(t-s)^{H- \frac 32}.$
\end{enumerate}
Then  $X$ is H\"older continuous of any order $a<H$.
\end{pro}

\subsection*{Self-Similar Processes} 

A process $X$ is \emph{self-similar} with index $\beta>0$ if
$$
(X_{at})_{0 \leq t \leq T/a}\overset{d}{=} (a^\beta X_t)_{0 \leq t \leq T},\quad  \mbox{for all }  a>0.
$$
In the Gaussian case this means that
$$
d_X(t,s) = a^{-\beta}d_X(at,as) \quad \mbox{ for all } a>0.
$$
So, it is clear that $X$ cannot be H\"older continuous of order $H>\beta$.

Let $\mathcal{H}_t^X$ be the closed linear subspace of $L^2(\Omega)$ generated by the Gaussian random variables $\{ X_s; s\le t\}$. Denote $\mathcal{H}_{0+}^X := \cap_{t\in(0,T]} \mathcal{H}_t^X$. Then $X$ is \emph{purely non-deterministic} if $\mathcal{H}_{0+}^X$ is trivial. By \cite{Yazigi} a purely non-deterministic Gaussian self-similar process admits the representation
\begin{equation}\label{rep:volterra}
X_t =  \int_0^t t^{\beta - \frac 12}F\left(\frac{u}{t}\right)\ud W_u,
\end{equation}
where $F \in L^2([0,1])$ is positive. Consequently: 

\begin{cor}\label{cor:ss-pnd}
Let $X$ be a purely non-deterministic  Gaussian self-similar process with index $\beta$ and representation \eqref{rep:volterra}. Then $X$ is H\"older continuous of any order $a<H$ if and only if
\begin{enumerate}
\item
$\int_s^t t^{2\beta-1}F(\frac{u}{t})^2\, \d u \le c_\epsilon|t-s|^{2H-\epsilon}$, 
\item
$\int_0^s \left|t^{\beta -\frac 12}F\left(\frac{u}{t}\right)-s^{\beta -\frac 12}F\left(\frac{u}{s}\right)\right|^2\, \d u \le c_\epsilon |t-s|^{2H-\epsilon}$ 
\end{enumerate}
for all $s<t$ and $\epsilon>0$.
\end{cor}

\begin{pro}
Let $X$ be a purely non-deterministic  Gaussian self-similar process with index $\beta$ and representation \eqref{rep:volterra}. Then $X$ is H\"older continuous of any order $a<H$ if 
\begin{enumerate}
\item $F\left(x\right) \leq c \,x^{\beta-H}(1-x)^{H- \frac 12},$ \;\; $0<x<1 ,$
\item $\left|1-\frac{F(x)}{F(y)}\right|\leq \left| \left(\frac yx\right)^{H-\beta}\left(\frac{1-x}{1-y}\right)^{H- \frac 12}-1\right|,$\;\; $0<y<x<1,$
\end{enumerate}
\end{pro}

\begin{proof} Condition (i) of Corollary \ref{cor:volterra} follows from assumption (i) and condition (ii) of Corollary \ref{cor:volterra} follows from assumption (i) and (ii) applied to the estimate
\begin{eqnarray*}
\lefteqn{\left|t^{\beta- \frac 12}F\left(\frac ut\right)-s^{\beta-\frac 12}F\left(\frac us\right)\right|} \\
&\le&
F\left(\frac ut\right) \,t^{\beta - \frac 12}\left|1-\frac{F\left(\frac us\right)}{F\left(\frac ut\right)}\right|+F\left(\frac us\right)\left|t^{\beta- \frac 12}-s^{\beta- \frac 12}\right|.
\end{eqnarray*}
The details are left to the reader.
\end{proof}


\end{document}